\documentclass[reqno,10pt]{amsart}
\usepackage{graphicx, color}
\usepackage{amsmath}
\usepackage{amssymb}
\usepackage{amsthm}
\usepackage{enumitem}
\usepackage{bbm}
\usepackage[l2tabu,orthodox]{nag}
\usepackage[all,error]{onlyamsmath}
\usepackage[strict]{csquotes}
\usepackage{url}

\newtheorem{theorem}{Theorem}[section]
\newtheorem{lemma}[theorem]{Lemma}
\newtheorem{corollary}[theorem]{Corollary}

\theoremstyle{definition}

\newtheorem{remark}[theorem]{Remark}

\numberwithin{equation}{section}

\title{Discrete Space-Time Wave Kernels and Trace Identities on Regular Graphs}

\author{Amar Ba\v{s}i\'{c}}
\address{University of Sarajevo, Faculty of Electrical Engineering, Zmaja od Bosne bb, Sarajevo, 71000, Bosnia and Herzegovina}
\email{abasic@etf.unsa.ba}
\urladdr{ORCID: \url{https://orcid.org/0000-0002-3213-4527}}

\author{Lejla Smajlovi\'c}
\address{University of Sarajevo, Department of Mathematics and Computer Science, Zmaja od Bosne 33-35, Sarajevo, 71000, Bosnia and Herzegovina}
\address{University of Sarajevo, School of Economics and Business, Trg oslobodjenja - Alija Izetbegovi\'{c} 1, Sarajevo, 71000, Bosnia and Herzegovina}
\email{lejlas@pmf.unsa.ba}
\urladdr{ORCID: \url{https://orcid.org/0000-0002-2709-5535}}

\author{Zenan \v{S}abanac}
\address{University of Sarajevo, Department of Mathematics and Computer Science, Zmaja od Bosne 33-35, Sarajevo, 71000, Bosnia and Herzegovina}
\email{zsabanac@pmf.unsa.ba}
\urladdr{ORCID: \url{https://orcid.org/0000-0001-8122-496X}}



\begin{document}

\begin{abstract}
We study the discrete space-time wave equation on a $(q+1)$-regular graph $X$ associated with the affine Laplace-type operator. For the forward time-difference scheme we derive explicit formulas for the two fundamental solutions (wave kernels) in terms of discrete modified Bessel functions and the non-backtracking walk counts on $X$ thus  providing a direct and explicit link between wave propagation and combinatorial graph data. Utilizing uniqueness property of the wave kernel, we prove a new trace-type formula associated to the affine Laplace-type operator on $X$ and apply it to deduce many combinatorial identities. For example, we derive a closed-form expression for evaluation of some trigonometric sums twisted by an additive character as well as evaluations of finite sums of Chebyshev polynomials twisted by binomial coefficients. 
\end{abstract}

\maketitle

\section{Introduction and main results}  

The wave equation on a simple, regular graph provides a fundamental framework for studying discrete wave propagation, bridging classical continuum mechanics with discrete structures. Physically, this equation models the dynamics of vibrations in complex mechanical networks, acoustic waveguides, and quantum particles that transition across lattice structures \cite{BO11, Sm07}, where the traditional spatial derivative is replaced by the graph Laplacian operator \cite{Ch97, FT04} or its linear perturbation \cite{MS99, Pe25} so that the bottom of its spectrum is equal to zero. See also  \cite{Ch22} for a study of properties of solutions of the semilinear wave equation on networks and \cite{CNT23, TS16} for recent advances related to the wave equation on trees in continuous time.

 Space-time discretization is a standard procedure in physics. An interested reader is referred to \cite{tH16} and the references therein for motivation stemming from quantum gravitational considerations. More recently, discrete space-time wave equations are used in the study of graph neural networks \cite{Yue+25} (the so called \textit{graph wave networks}). The discretization of time in \cite{Yue+25} proved to be particularly useful for numerical computations that are shown to be constantly stable. There are, of course, different possibilities for time discretization, the most common ones being the forward difference as in \cite{BSS26, Pe25}, the backward difference \cite{BSS24, KS23} or the symmetric timescale difference, see \cite{CP94, GNHS25}.

\subsection{Our setup}

In this paper we study the (generalized) discrete space-time wave equation 
\begin{equation}\label{wave_eq_disc}
\Delta_X^{a,b}W(x_0,x;t)+\partial_t^2 W(x_0,x;t)=0.
\end{equation}
on a connected, undirected $(q+1)$-regular graph $X$ with a finite or countably infinite set of vertices $V(X)$. Here  
\begin{equation*}\label{gen_Laplace}
    \Delta_{X}^{a,b}:= b\Delta_X - aI,
\end{equation*}
is the generalization of the combinatorial Laplacian $\Delta_X$, acting on functions $f:V(X)\to\mathbb{R}$, by
\begin{equation*}
    \Delta_X f(x) = (q+1)\, f(x) - \sum_{x \sim y} f(y).
\end{equation*}
We assume that the real parameters $a,\,b$ satisfy $b\neq0$ and the spectral-edge condition
\begin{equation}\label{eq:edge}
a\leq \begin{cases}
0, & \text{if } b>0 \\
-|b|\max\{\delta_{X={\textrm{finite}}}\lambda_{\max}, (\sqrt{q}+1)^2\}, & \text{if } b<0,
\end{cases}
\end{equation}
where $\lambda_{\max}$ is the largest eigenvalue of $\Delta_X$ if $X$ is finite ($\delta_{X={\textrm{finite}}}=1$ if $X$ is finite and $0$, otherwise). Trivially, \eqref{eq:edge} ensures non-negativity of $\Delta_{X}^{a,b}$ on a finite graph. Moreover, when $X$ is $(q+1)$-regular tree $T_{q+1}$, \eqref{eq:edge} yields the inequality $b(q+1)-a\geq 2|b|\sqrt{q}$, which ensures non-negativity of  $\Delta_{X}^{a,b}$  on $T_{q+1}$. 

Furthermore, $\partial_t$, for $t\in\ \mathbb{N}_0$, denotes the forward difference operator
\[
\partial_t u(t)=u(t+1)-u(t) \quad \text{and} \quad \partial_t^2 u(t)=u(t+2)-2u(t+1)+u(t).
\]
The generalized Laplacian $\Delta_{X}^{a,b}$ specializes to Laplace-type operators studied in \cite{AMPS13, BL13, CRS07, DS-PB18, MS99, Pe25} for appropriately chosen  values of $a,\,b$ satisfying \eqref{eq:edge}.

The \textit{discrete time wave kernels} associated with the operator $\Delta_X^{a,b}$ on the graph $X$ are two fundamental solutions $W_X^{a,b},\,V_X^{a,b}: V(X)\times V(X)\times\mathbb N_0 \to \mathbb R$ of \eqref{wave_eq_disc} satisfying the following  initial conditions:
\begin{equation}\label{cond1}
W_{X}^{a,b}(x_0, x; 0) =
\begin{cases}
1, & \text{if } x = x_0, \\
0, & \text{if } x \neq x_0,
\end{cases} \text{   and   } \partial_t W_{X}^{a,b}(x_0, x; 0) = 0,\,\, \forall x\in V(X),
\end{equation}
and
\begin{equation}\label{cond on V}
V_{X}^{a,b}(x_0, x; 0) = 0  \text{   and   } \partial_t  V_{X}^{a,b}(x_0, x; 0) =
\begin{cases}
1, & \text{if } x = x_0, \\
0, & \text{if } x \neq x_0,
\end{cases} \,\, \forall x\in V(X).
\end{equation}

\subsection{Main results}

Our first main result is an explicit evaluation of the discrete time wave kernels on any $(q+1)$-regular graph in terms of the combinatorial graph data and the discrete $I$-Bessel functions. The second main result is a trace-type formula which stems from the uniqueness of the wave kernel, by identifying the discrete time spectral realization of the wave kernel with its geometric/combinatorial evaluation.

\subsubsection{Wave kernels on regular graphs}

Let us first introduce some notation. 
Fix a base vertex $x_0\in V(X)$. For any $x \in V(X)$, let $c_m(x)$ be the number of non-backtracking walks of length $m$ from $x_0$ to $x$, as defined in \cite[Section~2.1]{CHJSV}. For any $x\in V(X)$ let
\begin{equation*}\label{eq:def-bm01}
    b_0(x)=c_0(x),\qquad b_1(x)=c_1(x).
\end{equation*}
For $m\geq 2$, let
\begin{equation}\label{eq:def-bm}
    b_m(x) := c_m(x)+(1-q)\big(c_{m-2}(x)+c_{m-4}(x)+\dots+c_{*}(x)\big),
\end{equation}
where
\begin{equation}\label{eq:def-cstar}
    c_{*}(x) :=
    \begin{cases}
        c_0(x), & \text{if $m$ is even}, \\
        c_1(x), & \text{if $m$ is odd}.
    \end{cases}
\end{equation}
If $X$ is vertex-transitive, let $N_m(x_0)$ denote the number of non-backtracking walks of length $m$ from $x_0$ to itself without tails (meaning there is no immediate backtracking at the beginning or end of the walk, see \cite[Section~2.1]{CHJSV}). Our first main result is the following theorem.
\begin{theorem}\label{wavegraph}
Let $X$ be a $(q+1)$-regular graph with a fixed base point $x_0$. Let $a\in\mathbb{R}$ and $b\in\mathbb R \setminus \{0\}$ be such that \eqref{eq:edge} holds, and set $c_{a,b} = \dfrac{2b\sqrt{q}}{b(q+1) - a}$ and $d_{a,b}=b(q+1)-a$.
Then the discrete-time wave kernels $W_X^{a,b}(x_0, x; t)$ and $V_X^{a,b}(x_0, x; t)$, associated with the operator $\Delta_X^{a,b}$ and satisfying the initial conditions \eqref{cond1} and \eqref{cond on V}, respectively, are given for $x \in V(X)$ and $t \in \mathbb{N}_0$ by
\begin{equation}\label{eq:universal_cover}
 W_X^{a,b}(x_0, x; t) =
    \sum_{k=0}^{\lfloor \tfrac{t}{2} \rfloor}
    (-1)^{k} \,\binom{t}{2k} \, d_{a,b}^k
    \sum_{m=0}^{k} (-1)^{m} b_m(x) \, q^{-\tfrac{m}{2}} \, I_{m}^{c_{a,b}}(k),
\end{equation}
\begin{equation}\label{eq:universal_cover V}
    V_X^{a,b}(x_0, x; t) =
    \sum_{k=0}^{\lfloor \tfrac{t-1}{2} \rfloor}
    (-1)^{k} \,\binom{t}{2k+1} \, d_{a,b}^k
    \sum_{m=0}^{k}(-1)^{m} b_m(x) \, q^{-\tfrac{m}{2}} \, I_{m}^{c_{a,b}}(k),
\end{equation}
where $I_{m}^{c_{a,b}}$ denotes the discrete $I$-Bessel function defined by \eqref{eq:def-I-Bessel} below.

Furthermore, if $x = x_0$ (so that $|x|=0$) and $X$ is vertex-transitive, then for all $t \in \mathbb{N}_0$ we have
\begin{multline}\label{eq:universal_cover_a}
        W_X^{a,b}(x_0, x_0; t) =
    \sum_{k=0}^{\lfloor t/2\rfloor}
    (-1)^{k} \,\binom{t}{2k} \, d_{a,b}^k
    \bigg( \sum_{m=0}^{k} (-1)^m N_m(x_0) \, q^{-\frac{m}{2}} \, I_{m}^{c_{a,b}}(k)
    \\+ (1-q) \sum_{j=1}^{\lfloor k/2 \rfloor} q^{-j} I_{2j}^{c_{a,b}}(k) \bigg),
\end{multline}
and
\begin{multline}\label{eq:universal_cover_a V}
    V_X^{a,b}(x_0, x_0; t) =
    \sum_{k=0}^{\lfloor (t-1)/2 \rfloor}
    (-1)^{k} \,\binom{t}{2k+1} \, d_{a,b}^k
    \bigg( \sum_{m=0}^{k} (-1)^m N_m(x_0) \, q^{-\tfrac{m}{2}} \, I_{m}^{c_{a,b}}(k)
    \\ +(1-q) \sum_{j=1}^{\lfloor k/2 \rfloor} q^{-j} I_{2j}^{c_{a,b}}(k) \bigg).
\end{multline}
\end{theorem}

\subsubsection{Trace-type formula on regular graphs}

Trace formulas on regular graphs, motivated by the Selberg trace formula on hyperbolic spaces \cite{Se56} have been extensively studied in the past 40 years, see \cite{Ah87, GLL26, Mn07, TW03, VN93}. In this paper we obtain a "physical" pre-trace formula which is based on utilizing the uniqueness of bounded solutions to the discrete wave equation in timescale $\mathbb T =\mathbb Z$ (see \cite[Theorem 3.1.]{Slav17}, or \cite{ABEPT03} and \cite{CS26} for a general timescale analogue). The pre-trace formula is then refined using the explicit representation of the coefficients $b_m(x)$ obtained in Lemma \ref{lemma:bm_from_cm} to deduce the following theorem. 

\begin{theorem} \label{th_trace_formula}
Let $X$ be a finite $(q+1)$-regular graph on $N=|V(X)|$ vertices with a fixed base point $x_0$ and let $0=\lambda_1\leq \ldots \leq \lambda_N=\lambda_{\max}$ be the eigenvalues of the combinatorial Laplacian $\Delta_X$ with the corresponding orthonormal eigenfunctions $\psi_1,\ldots,\psi_N$. Let $a\in\mathbb{R}$ and $b\in\mathbb R \setminus \{0\}$ be such that  \eqref{eq:edge} holds true. Then, with the notation as above, for any positive integer $k$ and any $x\in V(X)$ we have the following identity: 
\begin{equation}\label{eq. trace fla}
\begin{split}
        \sum_{j=1}^N& 
\mu_j^k 
\psi_j(x_0)\overline{\psi_j(x)} = d_{a,b}^k \sum_{m=0}^{k} (-1)^{m} b_m(x) \, q^{-\tfrac{m}{2}} \, I_{m}^{c_{a,b}}(k)\\
&=d_{a,b}^k \left( \delta_{x_0=x}I_0^{c_{a,b}}(k) + 2 \sum_{m=1}^{k} (-1)^{m} \Bigl[T_m\left(\frac{A}{2\sqrt{q}}\right)\delta_{x_0}\Bigr](x)\, I_{m}^{c_{a,b}}(k)\right),
\end{split}
\end{equation}
where $\mu_j=b\lambda_j-a\geq 0$, $j=1,\ldots, N$, $A$ is the adjacency matrix of $X$, $T_m$ is the Chebyshev polynomial of the first kind and $\delta_{x_0}$ is the column vector of $N$ elements with all entries equal to $0$, except for the entry corresponding to the vertex $x_0$, which is equal to $1$.  For a column vector $\mathbf a$ and $x\in V(X)$, $\mathbf a(x)$ denotes the entry of $\mathbf a$ corresponding to the vertex $x$.

\noindent Moreover, if $X$ is vertex-transitive, then 
\begin{equation}\label{eq. trace vert trans}
 \sum_{j=1}^N \mu_j^k =d_{a,b}^k \left( NI_0^{c_{a,b}}(k) + 2 \sum_{m=1}^{k} (-1)^{m}\mathrm{Tr} \Bigl[T_m\left(\frac{A}{2\sqrt{q}}\right)\Bigr]\, I_{m}^{c_{a,b}}(k)\right).
\end{equation}
\end{theorem}



\begin{remark}
 The trace-type formula \eqref{eq. trace fla} is of a similar type as Theorem C of \cite{GLL26}, meaning that it relates the spectrum of the Laplacian to the number of non-backtracking walks on a regular graph. Our trace formula does not involve a test function; however, it applies to a larger family of operators (generalized Laplacians $\Delta_X^{a,b}$). Moreover, it is of a different, "physical" origin - it stems from the wave propagation properties, so we view it as a physical trace-type formula. An interested reader is referred to \cite{CHJSV} for a trace formula stemming from the heat diffusion on a regular graph.   
\end{remark}

\subsection{Applications}

We present three applications of Theorem \ref{th_trace_formula}: to the discrete cycle on $n$ vertices, the complete graph on $n$ vertices and a $d$-dimensional Hamming cube. In all three cases, we deduce interesting combinatorial identities.

An application to the discrete cycle on $n$ vertices yields the following two identities.
For every $\ell\in\{1,\ldots,n-1\}$, every positive integer $k$, and every real number $c$, one has
\begin{equation}\label{bmCn_single_sum l}
    \frac{1}{n}\sum_{j=0}^{n-1}
e^{-2\pi i j\ell/n}
\left(
2c\sin^2\!\left(\frac{\pi j}{n}\right)+1-c
\right)^k  =
 \underset{n\mid m\pm \ell}{\sum_{m\in\{1,\ldots,k\}}} (-1)^m
I_m^{c}(k).
\end{equation}
Moreover,
\begin{equation}\label{eq:Cn_single_sum 0}
\frac{1}{n}\sum_{j=0}^{n-1}
\left(
2c\sin^2\!\left(\frac{\pi j}{n}\right)+1-c
\right)^k
=
I_0^{c}(k)
+
2\sum_{s=1}^{\lfloor k/n\rfloor}
(-1)^{s n}
I_{s n}^{c}(k).
\end{equation}

Trigonometric sums of the type \eqref{eq:Cn_single_sum 0} and \eqref{bmCn_single_sum l} appear in many physical and combinatorial settings. For example, they arise in describing a system of coupled oscillators under acceleration \cite{Ho19} or in one-dimensional Ising model with nearest neighbor interactions \cite[Chapter III]{MW14}. In this paper, we provide another "physical" interpretation/realization of these sums as a (pre-)trace of the wave kernel on a discrete cycle in discrete time which, in view of an explicit "geometric" expression for the kernel, yields their explicit evaluation.

In \cite[Remark 5.2]{BSS26}, it is shown that 
\begin{equation}\label{eq:I1-binomial}
   I_m^{1}(k)= 2^{-k}\binom{2k}{k-m}, \qquad 0\le m\le k.
\end{equation} 
Hence, \eqref{bmCn_single_sum l} and \eqref{eq:Cn_single_sum 0}, with $c=1$, coincide with the formula obtained in \cite[Corollary 11]{CJHSVcomb} for $\beta=0$,  and in this sense generalize the combinatorial results of \cite{dFK13,dFGK17,Me12}. See also \cite{JKS24} for a physical approach to secant and cosecant sums arising from heat diffusion.

Results similar to \eqref{eq:Cn_single_sum 0} and \eqref{bmCn_single_sum l} can be obtained by applying Theorem \ref{th_trace_formula} to any circulant graph. They will be studied in a follow-up paper.

An application to the complete graph on $n\geq 3$ vertices yields combinatorial identities relating Chebyshev polynomials and the discrete $I$-Bessel functions. For example, for every real number $c$, every $n\ge3$, and every $k \in \mathbb{N}$, we prove that
\begin{equation}\label{eq. kn illustration}
\sum_{m=1}^{k}
T_m\!\left(\frac{1}{2\sqrt{n-2}}\right)
I_m^{c}(k)= \frac{1}{2}\left(1+\frac{c}{2\sqrt{n-2}}\right)^k-\frac{1}{2}I_0^{c}(k).
\end{equation}
Taking $c=1$ in \eqref{eq. kn illustration}, using \eqref{eq:I1-binomial}, and multiplying by $2^{k+1}$ we obtain an identity
\begin{equation*}
2\sum_{m=1}^{k}
\binom{2k}{k-m}T_m\!\left(\frac{1}{2\sqrt{n-2}}\right)
= \left(2+\frac{1}{\sqrt{n-2}}\right)^k-\binom{2k}{k}.
\end{equation*}

Finally, when applying our results to the $d$-dimensional hypercube, we deduce identities involving Chebyshev polynomials, discrete $I$-Bessel functions and binomial coefficients. For example, Corollary~\ref{cor:Qd_twisted} below with $c=2(d-1)^{-1/2}$ yields the following identity, which holds for all $d\ge2$, $1\le r\le d$, and $k\in\mathbb N$:
\[
\sum_{j=0}^{d}
K_j(r;d)\left(\frac{2j-1}{d-1}\right)^k
=
2\sum_{m=1}^{k}
(-1)^m
\sum_{s=0}^{d}
K_s(r;d)
T_m\!\left(
\frac{d-2s}{2\sqrt{d-1}}
\right)
I_m^{\frac{2}{\sqrt{d-1}}}(k),
\]
where
\[
K_j(r;d)
=
\sum_{\ell=0}^{j}
(-1)^\ell
\binom{r}{\ell}
\binom{d-r}{j-\ell}
\]
is the $j$-th Krawtchouk polynomial. (Recall that, by definition, $\binom{n}{k}=0$ for all non-negative integers $k,\,n$ such that $k>n$.)

\section{Preliminaries}
In this section we introduce discrete Bessel functions, describe further notation needed in the sequel and prove a representation formula for the combinatorial quantities $b_m(x)$ related to numbers of non-backtracking walks of length $m$. 

\subsection{Discrete Bessel functions }
Discrete Bessel functions are introduced in \cite{BC} and further generalized and studied in \cite{BSS24, CHJSV, Cu15, CS26, KS23, Slav18} in discrete timescales, including forward and backward difference operators.  In this paper, we use the discrete $I$-Bessel function on the timescale $\mathbb Z$ associated with the forward difference operator.
This function can be expressed as
\begin{equation}\label{eq:def-I-Bessel}
I^{c}_{n}(t)
= \sum_{\ell=0}^{\lfloor (t-n)/2\rfloor}
\frac{t!}{\ell!\,(t-2\ell-n)!\,(n+\ell)!}
\left(\frac{c}{2}\right)^{2\ell+n},
\end{equation}
for $t,n\in\mathbb{N}_0$ with $n\le t$, and $c\in\mathbb{C}$. By convention, $I_n^c(t)=0$ for $n>t$.

We refer the interested reader to the above references for further properties of discrete Bessel functions. 

\subsection{Wave kernels on a $(q+1)$-regular tree}

The universal cover of any $(q+1)$-regular graph is a $(q+1)$-regular tree (also known as the Bethe lattice with coordination number $q+1$), which is a connected $(q+1)$-regular graph with no cycles. We denote such a tree by $T_{q+1}$.
In this section, we recall the main result of \cite{BSS26} in which the wave kernels on 
$T_{q+1}$ are explicitly computed. Given that $T_{q+1}$ is distance-transitive, the kernels depend only on the radial distance $r=d(x_0,x)=|x|$ from the root $x_0$. The solutions $W_{q+1}^{a,b}(r;t)$ and $V_{q+1}^{a,b}(r;t)$ of the wave equation \eqref{wave_eq_disc} satisfying the initial conditions \eqref{cond1} and \eqref{cond on V} are given by
\begin{multline}\label{eq:tree_W}
W_{q+1}^{a,b}(r;t)
=\sum_{k=0}^{\left\lfloor t/2 \right\rfloor}
(-1)^{k+r}\binom{t}{2k} d_{a,b}^k
\bigg(
q^{-\frac{r}{2}} I_r^{c_{a,b}}(k)\\
-(q-1)\sum_{\ell=1}^{\left\lfloor (k-r)/2 \right\rfloor}
q^{-\frac{r+2\ell}{2}} I_{r+2\ell}^{c_{a,b}}(k)
\bigg),
\end{multline}
and
\begin{multline*}\label{eq:tree_V}
V_{q+1}^{a,b}(r;t)
=\sum_{k=0}^{\left\lfloor (t-1)/2 \right\rfloor}
(-1)^{k+r}\binom{t}{2k+1} d_{a,b}^k
\bigg(
q^{-\frac{r}{2}} I_r^{c_{a,b}}(k)\\
-(q-1)\sum_{\ell=1}^{\left\lfloor (k-r)/2 \right\rfloor}
q^{-\frac{r+2\ell}{2}} I_{r+2\ell}^{c_{a,b}}(k)
\bigg),
\end{multline*}
where $c_{a,b} = \dfrac{2b\sqrt{q}}{b(q+1) - a}$ and $d_{a,b}=b(q+1)-a$. These expressions provide the fundamental building blocks for our analysis on general $(q+1)$-regular graphs.

\subsection{Graph coverings}

Recall that a graph $\widetilde X$ is called a covering graph of $X$ if there exists a map
\[
\pi:V(\widetilde X)\to V(X)
\]
which locally preserves adjacency: for every vertex $\widetilde x\in V(\widetilde X)$, the neighbors of $\widetilde x$ are mapped bijectively onto the neighbors of $\pi(\widetilde x)$ in $X$. 

In particular, if $X$ is a connected $(q+1)$-regular graph, then its universal cover is the $(q+1)$-regular tree $T_{q+1}$ with covering map
\[
\pi:V(T_{q+1})\to V(X),
\]
see sections 5 and 6 of \cite{CY99} for more details. 

Fix a lift $\widetilde x_0\in V(T_{q+1})$ of the base vertex $x_0\in V(X)$. The covering map $\pi$ preserves adjacency locally, and therefore, for every $r\ge0$ and every $x\in V(X)$, the number of vertices $\widetilde x\in\pi^{-1}(x)$ satisfying
\[
d_{T_{q+1}}(\widetilde x_0,\widetilde x)=r
\]
is equal to the number $c_r(x)$ of non-backtracking walks in $X$ from $x_0$ to $x$ of length $r$.

\subsection{A representation formula for $b_m(x)$}

The number $c_m(x)$ of non-backtracking walks of length $m$ from $x_0$ to $x$ is a well studied object in graph theory. Its generating function was deduced in \cite{Fr91}, the exponential generating function was deduced and studied in \cite{AGHN18} (see also \cite{AHNW24} for a corresponding weighted graph result and \cite{OS14} for the distributional result). However, despite an extensive literature search we could not find an expression for the coefficients $b_m(x)$, though it might be folklore in graph theory. For this reason, we derive an explicit expression for the coefficients $b_m(x)$ directly in terms of the adjacency operator and Chebyshev polynomials.

Recall that $\delta_{x_0}$ is the column vector of $|V(X)|$ elements with zeros at all positions except at the one corresponding to $x_0$, where it is equal to $1$, and that for any column vector $\mathbf a$ we denote by $\mathbf a(x)$ the entry corresponding to the vertex $x$. With this notation, we have the following lemma.


\begin{lemma}\label{lemma:bm_from_cm}
Let $X$ be a finite $(q+1)$-regular graph with adjacency matrix $A$ and a fixed root
$x_0\in V(X)$. Let $c_m(x)$ be the number of non-backtracking walks of length $m$
from $x_0$ to $x$. Then, for $m\ge1$
\begin{equation*}\label{eq. bm as Cheb}
b_m(x)
=
\left[
2q^{m/2}T_m\!\left(\frac{A}{2\sqrt q} \right)\delta_{x_0}
\right](x),
\end{equation*}
where $T_m(x)$ are Chebyshev polynomials of the first kind.
\end{lemma}

\begin{proof}
For each $m\ge0$, let $\mathbf c_m=(c_m(x))_{x\in V(X)}$
denote the column vector whose $x$-th component is the number of non-backtracking walks of length $m$ from $x_0$ to $x$. Then
\[
\mathbf c_0=\delta_{x_0},
\qquad
\mathbf c_1=A\delta_{x_0}.
\]
Moreover, the vectors $\mathbf c_m$ satisfy the standard recurrence for
non-backtracking walks on a $(q+1)$-regular graph (see
\cite[Section~2.1]{CHJSV}):
\[
\mathbf c_2=A\mathbf c_1-(q+1)\mathbf c_0,
\]
and, for $m\ge3$,
\[
\mathbf c_m
=
A\mathbf c_{m-1}
-
q\mathbf c_{m-2}.
\]

\noindent
For $m\ge 0$, let $\mathbf b_m$ denote the column vector whose $x$-th component is $b_m(x)$.
Substituting the recurrence for $\mathbf c_m$ into the definition of
$\mathbf b_m$ and collecting terms gives
\[
\mathbf b_m
=
A\mathbf b_{m-1}
-
q\mathbf b_{m-2},
\qquad m\ge3,
\]
with initial values
\[
\mathbf b_1=A\delta_{x_0},
\qquad
\mathbf b_2=(A^2-2qI)\delta_{x_0}.
\]
On the other hand, using the classical recurrence relation $T_m(x)=2xT_{m-1}(x)-T_{m-2}(x)$  \cite[formula 8.941.1]{GR07} for the Chebyshev polynomials of the first kind with
$x=s/(2\sqrt q)$, we obtain that
\[
\begin{split}
P_m(s)
&:=2q^{m/2}T_m\!\left(\frac{s}{2\sqrt q}\right)=2q^{m/2}
\left[
2\frac{s}{2\sqrt q}
T_{m-1}\!\left(\frac{s}{2\sqrt q}\right)
-
T_{m-2}\!\left(\frac{s}{2\sqrt q}\right)
\right]\\
&=sP_{m-1}(s)-qP_{m-2}(s).
\end{split}
\]
Moreover, since $T_1(x)=x$ and $T_2(x)=2x^2-1$, we have
\[
P_1(s)=s,
\qquad
P_2(s)=s^2-2q.
\]
Substituting $s=A$ and multiplying by $\delta_{x_0}$ yields
\[
P_m(A)\delta_{x_0}
=
A\,P_{m-1}(A)\delta_{x_0}
-
q\,P_{m-2}(A)\delta_{x_0}.
\]
This proves that the sequence $\mathbf u_m:=P_m(A)\delta_{x_0}$ satisfies the recurrence
\[
\mathbf u_m
=
A\mathbf u_{m-1}
-
q\mathbf u_{m-2},\quad m\ge 3.
\]
Moreover,
\[
\mathbf u_1
=
P_1(A)\delta_{x_0}
=
A\delta_{x_0}
=
\mathbf b_1,
\]
and
\[
\mathbf u_2
=
P_2(A)\delta_{x_0}
=
(A^2-2qI)\delta_{x_0}
=
\mathbf b_2.
\]
Therefore the sequences $\{\mathbf u_m\}$ and $\{\mathbf b_m\}$
satisfy the same second-order linear recurrence and have identical
initial values. By uniqueness of solutions of the recurrence, it follows that
\[
\mathbf b_m=P_m(A)\delta_{x_0},
\qquad m\ge1.
\]
The proof is complete.
\end{proof}

For vertex-transitive graphs, combining Lemma~\ref{lemma:bm_from_cm} with a result
of \cite{CHJSV} yields the following formula.

\begin{corollary}\label{cor:Nm_from_bm}
Let $X$ be a finite vertex-transitive $(q+1)$-regular graph. Then, for every
$m\ge2$,
\begin{equation}\label{eq. Nm at x0}
N_m(x_0)
=
\left[
2q^{m/2}
T_m\!\left(\frac{A}{2\sqrt q}\right) \delta_{x_0}
\right](x_0)
+
(q-1)\frac{1+(-1)^m}{2}.
\end{equation}
Moreover, if $X$ is also bipartite then
\[
\operatorname{Tr}
T_{2m+1}\!\left(\frac{A}{2\sqrt q}\right)=0,
\]
for all non-negative integers $m$.
\end{corollary}

\begin{proof}
By \cite[Lemma~2.3(3)]{CHJSV}, one has $b_0(x_0)=N_0(x_0)=1$, and for $m\ge1$,
\begin{equation}\label{eq:relation_b_m}
b_m(x_0) =
\begin{cases}
    N_m(x_0) + (1 - q), & \text{if $m$ is even}, \\
    N_m(x_0), & \text{if $m$ is odd},
\end{cases}
\end{equation}
which is equivalent to
\[
N_m(x_0)
=
b_m(x_0)
+
(q-1)\frac{1+(-1)^m}{2},\quad m\geq 1.
\]
Now, \eqref{eq. Nm at x0} follows from Lemma~\ref{lemma:bm_from_cm} with $x=x_0$.

If $X$ is bipartite, every closed walk, and hence every closed non-backtracking walk without tails, on $X$ has
even length. Therefore $N_{2m+1}(x_0)=0$ for $m\ge0$. Since $X$ is vertex-transitive, every diagonal entry of
$T_m(A/(2\sqrt q))$ is the same, hence
\begin{equation}\label{eq. trace mnev}
\left[T_m\!\left(\frac{A}{2\sqrt q}\right)\delta_{x_0}\right](x_0)
=
\frac1{|V(X)|}
\operatorname{Tr}
T_m\!\left(\frac{A}{2\sqrt q}\right).
\end{equation}
Thus,  $N_{2m+1}(x_0)=0$ combined with \eqref{eq. Nm at x0} proves the second statement.
\end{proof}

\begin{remark}
For vertex-transitive graphs, Corollary~\ref{cor:Nm_from_bm} is consistent
with Corollary~2 of \cite{Mn07}. More precisely, Corollary~2 of \cite{Mn07} gives the total number $\operatorname{\textbf gp}_m$ of closed non-backtracking walks of length $m\ge 1$ without tails on a regular graph as
\begin{equation}\label{geodesics}
\begin{split}
\operatorname{\textbf gp}_m
&=
2q^{m/2}
\sum_{j=1}^{|V(X)|}
T_m\!\left(\frac{\theta_j}{2\sqrt q}\right)
+
\frac{1+(-1)^m}{2}(q-1)|V(X)|\\
&=2q^{m/2}\operatorname{Tr}
T_m\!\left(\frac{A}{2\sqrt q}\right)+
\frac{1+(-1)^m}{2}(q-1)|V(X)|,
\end{split}
\end{equation}
where $\theta_j$, $j=1,\ldots, |V(X)|$ are the eigenvalues of the adjacency matrix.

If the graph is vertex-transitive, for $m\ge 1$ we have that $\operatorname{\textbf gp}_m= |V(X)| N_m(x_0)$, hence \eqref{geodesics} may be rewritten as
\begin{equation*}\label{eq. N0 as trace}
N_m(x_0)
=
\frac{2q^{m/2}}{|V(X)|}
\operatorname{Tr}
T_m\!\left(\frac{A}{2\sqrt q}\right)
+
(q-1)\frac{1+(-1)^m}{2}.
\end{equation*}
In view of \eqref{eq. trace mnev}, we deduce that Mn\"ev's formula \eqref{geodesics} is equivalent to \eqref{eq. Nm at x0}.
\end{remark}

\begin{remark}
It is well known that formulas for the numbers of closed non-backtracking walks on regular graphs can also be obtained from the
Ihara--Bass determinant formula and the spectral theory of the
non-backtracking (Hashimoto) matrix; see, for example,
\cite{AFH15, FP23, GK21}. The approach developed in this section is different,
since it derives these quantities directly from the recurrence relation for non-backtracking walks and the resulting Chebyshev representation.
\end{remark}

\section{Proof of Theorem \ref{wavegraph}} \label{sec. proof main thm}

Let $X$ be any $(q+1)$-regular graph, and let $T_{q+1}$ denote its universal cover with the covering map $\pi$. Fix a lift $\tilde{x}_0\in T_{q+1}$ of the vertex $x_0\in X$. This covering map allows us to relate the wave kernels $W_X^{a,b}(x_0, x; t)$ and $V_X^{a,b}(x_0, x; t)$ on $X$ to the wave kernels on $T_{q+1}$ via
\begin{equation} \label{eq:cover_sum}
    W_X^{a,b}(x_0, x; t) = \sum_{\tilde{x} \in \pi^{-1}(x)} W^{a,b}_{q+1}(\tilde{x}_0, \tilde{x}; t), \quad  V_X^{a,b}(x_0, x; t) = \sum_{\tilde{x} \in \pi^{-1}(x)} V^{a,b}_{q+1}(\tilde{x}_0, \tilde{x}; t).
\end{equation}

The number of vertices $\tilde{x}\in\pi^{-1}(x)$ at distance $r$ from $\tilde{x}_0$ equals the number $c_r(x)$ of non-backtracking walks of length $r$ from $x_0$ to $x$ in $X$. In other words,   $W^{a,b}_{q+1}(\tilde{x}_0, \tilde{x}; t) = W^{a,b}_{q+1}(r; t)$  and $V^{a,b}_{q+1}(\tilde{x}_0, \tilde{x}; t) = V^{a,b}_{q+1}(r; t)$ for exactly $c_r(x)$ vertices $\tilde{x} \in \pi^{-1}(x)$, hence \eqref{eq:cover_sum} becomes
\begin{equation} \label{eq:series_form}
    W_X^{a,b}(x; t) = \sum_{r \geq 0} c_r(x) \, W^{a,b}_{q+1}(r; t), \quad  V_X^{a,b}(x; t) = \sum_{r \geq 0} c_r(x) \, V^{a,b}_{q+1}(r; t),
\end{equation}
where we write $W_X^{a,b}(x; t) := W_X^{a,b}(x_0, x; t)$ and $V_X^{a,b}(x; t) := V_X^{a,b}(x_0, x; t)$ since $x_0$ is fixed.
Note that the sum in \eqref{eq:series_form} is finite for each $t \in \mathbb{N}_0$, because $W^{a,b}_{q+1}(r; t)=V^{a,b}_{q+1}(r; t) = 0$ whenever $r > t$.

We now prove equation \eqref{eq:universal_cover}; the proof of \eqref{eq:universal_cover V} is analogous.
Equation \eqref{eq:tree_W} yields
\begin{multline*}
    W_X^{a,b}(x; t) = \sum_{r \geq 0} c_r(x) \sum_{k=0}^{\infty} (-1)^{k+r} \binom{t}{2k}d_{a,b}^k
    \bigg[ q^{-\tfrac{r}{2}} I_r^{c_{a,b}}(k) \\
    - (q-1) \sum_{\ell=1}^{\lfloor (k-r)/2 \rfloor} q^{-\tfrac{r+2\ell}{2}} I_{r+2\ell}^{c_{a,b}}(k) \bigg].
\end{multline*}
Fix $m \in \mathbb{N}_0$. For a given $k \ge m$, the term $I_m^{c_{a,b}}(k)$ in the above sum appears when $m = r$, as well as for all $\ell \in \{1, \dots, \lfloor (k-r)/2 \rfloor\}$ such that $r + 2\ell = m$.
This means that $I_m^{c_{a,b}}(k)$ collects contributions from walks of length $m$ that are either non-backtracking walks of length $m$ or are obtained from shorter non-backtracking walks by adding $2\ell$ extra steps.

Hence, the coefficient of $I_m^{c_{a,b}}(k)$ is given by
\[
    b_m(x) = c_m(x) + (1-q)\big( c_{m-2}(x) + c_{m-4}(x) + \dots + c_*(x) \big),
\]
where $b_m(x)$ and $c_*(x)$ are as in \eqref{eq:def-bm}–\eqref{eq:def-cstar}.

Finally, if $X$ is $(q+1)$-regular and vertex-transitive, combining \eqref{eq:universal_cover} and \eqref{eq:relation_b_m} yields \eqref{eq:universal_cover_a}. The proof of \eqref{eq:universal_cover_a V} is analogous, so we omit it.
\qed

\section{Proof of Theorem \ref{th_trace_formula}}

We start by recalling that, by the spectral theorem for self-adjoint operators (see, e.g.,
\cite[Chapter~1]{Ch97}), the graph Laplacian admits an orthonormal basis
of eigenfunctions $\{\psi_j\}$ with the corresponding eigenvalues
$\{\lambda_j\}$. Consequently, 
\[
\Delta_X^{a,b}
=
\sum_j (b\lambda_j-a)\,\psi_j\psi_j^T.
\]

For $x,y\in V(X)$ let us define

\begin{equation}\label{eq. wave spectral}
\mathcal W_X^{a,b}(x,y;t)
=
\frac12
\sum_j
\left[
\left(1+i\sqrt{\mu_j}\right)^t
+
\left(1-i\sqrt{\mu_j}\right)^t
\right]
\psi_j(x)\overline{\psi_j(y)}.
\end{equation}
(Recall that, by \eqref{eq:edge} we have $\mu_j\geq 0$, $j=1,\ldots,N$).
We have
\begin{equation*}
\begin{split}
    \partial_t^2 \mathcal W_X^{a,b}(x_0,x;t) =& - \frac12\sum_j\mu_j\left[
\left(1+i\sqrt{\mu_j}\right)^t
+
\left(1-i\sqrt{\mu_j}\right)^t
\right] \psi_j(x_0)\overline{\psi_j(x)}\\
&= -\Delta_X^{a,b}\mathcal W_X^{a,b}(x_0,x;t),
\end{split}
\end{equation*}
where we used the fact that $\psi_j$ is the  eigenfunction of $\Delta_X^{a,b}$ with the real eigenvalue $\mu_j$. Therefore, $\mathcal W_X^{a,b}(x_0,x;t)$ satisfies the wave equation \eqref{wave_eq_disc}. Moreover, $\mathcal W_X^{a,b}(x_0,x;0)= \sum\limits_j
\psi_j(x_0)\overline{\psi_j(x)}=\delta_{x_0=x}$. Trivially,
\begin{equation*}
\partial_t\mathcal W_X^{a,b}(x_0,x;t)\big|_{t=0} = \frac{i}2
\sum_j
\sqrt{\mu_j}\left[
\left(1+i\sqrt{\mu_j}\right)^0
-
\left(1-i\sqrt{\mu_j}\right)^0
\right]
\psi_j(x_0)\overline{\psi_j(x)}=0,
\end{equation*}
hence $\mathcal W_X^{a,b}(x_0,x;t)$ also fulfills initial conditions \eqref{cond1}. From the uniqueness of bounded solutions to the discrete wave equation in timescale $\mathbb T =\mathbb Z$ (see \cite[Theorem 3.1.]{Slav17}, or \cite{ABEPT03} and \cite{CS26} for a general timescale analogue) we deduce that $\mathcal W_X^{a,b}(x_0,x;t) = W_X^{a,b}(x_0,x;t)$ for all $x\in V(X)$.

By identifying \eqref{eq. wave spectral} with \eqref{eq:universal_cover} and applying the binomial theorem to the spectral side we get 
\begin{equation*}\label{trace_identity}
\begin{split}
    W_X^{a,b}(x_0,x;t)&=\sum_{k=0}^{\lfloor t/2\rfloor}
(-1)^k\binom{t}{2k}
\sum_{j=1}^{N}
\mu_j^k
\psi_j(x_0)\overline{\psi_j(x)}\\&=
    \sum_{k=0}^{\lfloor t/2 \rfloor}
    (-1)^{k} \,\binom{t}{2k} \, d_{a,b}^k
    \sum_{m=0}^{k} (-1)^{m} b_m(x) \, q^{-\tfrac{m}{2}} \, I_{m}^{c_{a,b}}(k),
\end{split}
\end{equation*}
for all integers $t\geq 0$. The above identity is the identity between two binomial transforms. In view of the fact that the binomial transform is an involution, applying Lemma \ref{lemma:bm_from_cm} for $m\geq 1$ and using that $b_0(x)=\delta_{x=x_0}$, we immediately deduce \eqref{eq. trace fla}. 

\medskip

When $X$ is vertex-transitive, we take $x = x_0$ in \eqref{eq. trace fla} and sum over all $x_0\in V(X)$. Given that the basis $\{\psi_j\}_{j=1}^N$ is orthonormal, the left-hand side of \eqref{eq. trace fla} becomes $\sum_{j=1}^N \mu_j^k$. Using \eqref{eq. trace mnev}, it is trivial to see that the sum over $x=x_0\in V(X)$ of the right-hand side of \eqref{eq. trace fla} equals the right-hand side of \eqref{eq. trace vert trans}. The proof is complete. \qed

\section{Applications of the trace-type formula}

In this section, we apply Theorem \ref{th_trace_formula} to three classes of finite vertex-transitive graphs: the cycle graph $C_n$, the complete graph $K_n$, and the $d$-dimensional hypercube $Q_d$. The spectral side of the trace type formula in those three cases is known. Using Lemma  \ref{lemma:bm_from_cm} we will be able to express coefficients $b_m(x)$ and $N_m(x_0)$ in a closed form and derive interesting identities.

\subsection{Cycle $C_n$}

For the cycle graph $C_n$ on $n$ vertices, the coefficients $b_m(x_\ell)$ and the numbers $N_m(x_0)$ can be obtained directly from the definition of non-backtracking walks. We label the vertices of $C_n$ by $x_j=j$, $
j=0,\ldots, n-1$. Trivially, $C_n$ is $2$-regular, hence $q=1$. A non-backtracking walk, once it chooses an initial direction, either clockwise or counterclockwise, is forced to continue in that direction, since the only alternative at each step would be an immediate backtracking move, which is forbidden.

Since $q=1$ for $C_n$, the correction term in \eqref{eq:def-bm} vanishes.
Thus, for every $x_\ell\in V(C_n)$, we have
\[
b_m(x_\ell)=c_m(x_\ell).
\]
Moreover, a non-backtracking walk from $x_0$ to $x_\ell$ is forced to move
in one of the two directions around the cycle. Hence
\begin{equation}\label{bm_Cn}
    b_m(x_\ell)=
\begin{cases}
1, & m=0,\ \ell=0,\\
0, & m=0,\ \ell\ne0,\\
\mathbf 1_{m\equiv \ell\;(\mathrm{mod}\,n)}
+
\mathbf 1_{m\equiv -\ell\;(\mathrm{mod}\,n)}, & m\ge 1.
\end{cases}
\end{equation}
In particular, for $x_\ell=x_0$, this gives the number of closed
non-backtracking walks based at $x_0$. Therefore, a closed non-backtracking
walk returns to $x_0$ if and only if its length is a multiple of $n$.
In that case there are exactly two such walks, one in each direction, and no others.
Together with the convention $N_0(x_0)=1$, we obtain
\begin{equation} \label{C_n_formula}
    N_m(x_0)=
    \begin{cases}
1, & m=0,\\
2, & n\mid m,\ m\ge 1,\\
0, & \text{otherwise}.
\end{cases}
\end{equation}
Note that, for $\ell=0$, the formula \eqref{bm_Cn} reduces exactly to \eqref{C_n_formula}; hence $b_m(x_0)=N_m(x_0)$ for all $m\ge0$.

The following corollary gives an explicit representation of the discrete-time wave kernel on $C_n$.

\begin{corollary}\label{cor:Cn}
The discrete-time wave kernel on $C_n$ satisfying the initial condition \eqref{cond1} for all $t\in\mathbb N_0$ and $a\in\mathbb R$, $b\in\mathbb R\setminus \{0\}$ satisfying \eqref{eq:edge} is given by
\begin{equation}\label{W_C_n}
\begin{split}
W_{C_n}^{a,b}(x_0,x_0;t)
=
\sum_{k=0}^{\lfloor t/2 \rfloor}
(-1)^k
\binom{t}{2k}
(2b-a)^k
\Bigg(
I_0^{\frac{2b}{2b-a}}(k)
+
2\sum_{j=1}^{\lfloor k/n \rfloor}
(-1)^{jn}
I_{jn}^{\frac{2b}{2b-a}}(k)
\Bigg),
\end{split}
\end{equation}
and
\begin{equation}\label{W_C_n at x_l}
\begin{split}
W_{C_n}^{a,b}(x_0,x_\ell;t)
=
\sum_{k=0}^{\lfloor t/2 \rfloor}
(-1)^k
\binom{t}{2k}
(2b-a)^k
\underset{n\mid m\pm \ell}{\sum_{m\in\{1,\ldots,k\}}}
(-1)^{m}
I_{m}^{\frac{2b}{2b-a}}(k),
\end{split}
\end{equation}
for $\ell\in\{1,\ldots,n-1\}$.
\end{corollary}

\begin{proof}
Using
\eqref{C_n_formula} in Theorem~\ref{wavegraph} with $q=1$  yields \eqref{W_C_n}.

For $\ell\in\{1,\ldots,n-1\}$,
Theorem~\ref{wavegraph} gives
\[
W_{C_n}^{a,b}(x_0,x_\ell;t)
=
\sum_{k=0}^{\lfloor t/2 \rfloor}
(-1)^k
\binom{t}{2k}
(2b-a)^k
\sum_{m=0}^{k}
(-1)^m
b_m(x_\ell)
I_m^{\frac{2b}{2b-a}}(k).
\]
Since $\ell\ne0$, we have $b_0(x_\ell)=0$. For $m\ge1$, by
\eqref{bm_Cn},
only the indices $m$ that satisfy
$m\equiv \ell\pmod n$ or $m\equiv -\ell\pmod n$ contribute to the sum.
Therefore,
\begin{equation}\label{eq. bm circle sum}
\sum_{m=0}^{k}
(-1)^m
b_m(x_\ell)
I_m^{\frac{2b}{2b-a}}(k)
=
\underset{n\mid m\pm \ell}{\sum_{m\in\{1,\ldots,k\}}}
(-1)^m
I_m^{\frac{2b}{2b-a}}(k).
\end{equation}
Substituting this into the preceding identity yields
\eqref{W_C_n at x_l}.

\end{proof}

The explicit formulas obtained in Corollary~\ref{cor:Cn} may now be combined with the Theorem \ref{th_trace_formula} in order to derive spectral identities for the cycle graph $C_n$.

\begin{corollary} \label{cor. trig sums gen}
Let $C_n$ be the cycle graph and 
$\ell\in\{1,\ldots,n-1\}$. Then, for every $k\in\mathbb N_0$ and every real number $c$ the
identities \eqref{bmCn_single_sum l} and \eqref{eq:Cn_single_sum 0} hold. 

\end{corollary}

\begin{proof}
For the cycle graph $C_n$, we have $|V(C_n)|=n$ and $q=1$. The normalized eigenfunctions of the combinatorial Laplacian are $\psi_j(x_r)=\frac1{\sqrt n}e^{2\pi i jr/n}$, $j,r=0,\ldots,n-1$,
hence
\[
\psi_j(x_0)\overline{\psi_j(x_\ell)}
=
\frac1n e^{-2\pi i j\ell/n}.
\]
Moreover, the eigenvalues of the combinatorial graph Laplacian on $C_n$ are
$\lambda_j=4\sin^2\!\left(\frac{\pi j}{n}\right)$, $j=0,\ldots,n-1$; hence the eigenvalues of $\Delta_{C_n}^{a,b}=b\Delta_{C_n}-aI$ are
\[
\mu_j=
4b\sin^2\!\left(\frac{\pi j}{n}\right)-a, \qquad j=0,\ldots,n-1.
\]
Let $a\in\mathbb R$ and $b\in\mathbb R\setminus\{0\}$ be such that
\eqref{eq:edge} holds. From \eqref{eq. trace fla} combined with \eqref{bm_Cn} and \eqref{eq. bm circle sum}, by setting $c=\frac{2b}{2b-a}$ we obtain \eqref{bmCn_single_sum l}.
The identity \eqref{eq:Cn_single_sum 0} follows from the same argument with
$x=x_0$, using \eqref{C_n_formula} in place of \eqref{bm_Cn}.

This proves \eqref{bmCn_single_sum l} and \eqref{eq:Cn_single_sum 0} for  $c=\frac{2b}{2b-a}$ where $a,\, b\neq 0$ satisfy \eqref{eq:edge}. In view of \eqref{eq:def-I-Bessel}, both sides of identities
\eqref{bmCn_single_sum l} and \eqref{eq:Cn_single_sum 0} are polynomials
in $c$ which are equal for infinitely many values of $c$, hence they are equal for all $c\in\mathbb R$ by the identity theorem for polynomials.
\end{proof}

Taking $c=1$ in \eqref{bmCn_single_sum l} and \eqref{eq:Cn_single_sum 0} and using \eqref{eq:I1-binomial} we immediately deduce the following corollary.

\begin{corollary}
For every $k\in\mathbb N$ and $\ell=1,\ldots,n-1$, one has
\[
\sum_{j=0}^{n-1}
e^{-2\pi i j\ell/n}
\sin^{2k}\!\left(\frac{\pi j}{n}\right)
=
\frac{n}{2^{2k}} \underset{n\mid m\pm \ell}{\sum_{m\in\{1,\ldots,k\}}}(-1)^m
\binom{2k}{k-m}.
\]
Moreover, for every $k\in\mathbb N$, one has
\[
\sum_{j=0}^{n-1}\sin^{2k}\!\left(\frac{\pi j}{n}\right)
=
\frac{n}{2^{2k}} 
\left(
\binom{2k}{k}
+
2\sum_{\ell=1}^{\lfloor k/n\rfloor}
(-1)^{\ell n}\binom{2k}{k-\ell n}
\right).
\]
\end{corollary}

\subsection{Complete graph $K_n$}

Let $n\ge 3$. For the complete graph $K_n$ we have
\[
|V(K_n)|=n,
\qquad
\deg(x_0)=n-1,
\qquad
q=n-2.
\]
The adjacency matrix of $K_n$ is $A=J-I$, where $J$ denotes the $n\times n$ matrix with all entries equal to one and $I$ is the $n$-dimensional identity matrix. The spectrum of the adjacency matrix is given by
\[
\sigma(A)=\{n-1\ (\text{multiplicity }1),\ -1\ (\text{multiplicity }n-1)\};
\]
see, for example, \cite[Chapter~1]{Ch97}. 

In the following lemma we give explicit formulas for $b_m(x_\ell)$ and
$N_m(x_0)$ on the complete graph $K_n$ in terms of the Chebyshev polynomials $T_m$.
\begin{lemma}\label{cor:Kn_bm_Nm}
Let $K_n$ be the complete graph with vertices
$x_0,x_1,\ldots,x_{n-1}$. Then, $b_0(x_\ell)=1$ if and only if $\ell=0$ and equal to zero, otherwise. For $m\ge1$ we have that
\begin{equation}\label{eq. bm complete graph}
b_m(x_\ell)
=
\frac{2(n-2)^{m/2}}{n}
\left[
T_m\!\left(\frac{n-1}{2\sqrt{n-2}}\right)
+
\eta_\ell
T_m\!\left(-\frac{1}{2\sqrt{n-2}}\right)
\right],
\end{equation}
where
\[
\eta_\ell=
\begin{cases}
n-1, & \ell=0,\\
-1, & \ell=1,\ldots,n-1,
\end{cases}
\]
and
\begin{equation}\label{eq. n0 complete graph}
\begin{split}
N_m(x_0)
&=
\frac{2(n-2)^{m/2}}{n}
\left[
T_m\!\left(\frac{n-1}{2\sqrt{n-2}}\right)
+
(n-1)T_m\!\left(-\frac{1}{2\sqrt{n-2}}\right)
\right]  \\
&\qquad
+
\frac{1+(-1)^m}{2}(n-3).
\end{split}
\end{equation}
\end{lemma}

\begin{proof}
The spectral projection decomposition of the adjacency matrix $A$ is
\[
A=(n-1)P_0-P_1,
\]
where
\[
P_0=\frac1nJ,
\qquad
P_1=I-\frac1nJ
\]
are orthogonal projection matrices. Hence, for
any polynomial $p$,
\[
p(A)=p(n-1)P_0+p(-1)P_1.
\]
By Lemma~\ref{lemma:bm_from_cm} we have 
\begin{align*}
b_m(x)&=
\left[
2(n-2)^{m/2}T_m\!\left(\frac{A}{2\sqrt{n-2}} \right)\delta_{x_0}
\right](x)\\&=2(n-2)^{m/2}\left[
T_m\!\left(\frac{n-1}{2\sqrt{n-2}} \right)P_0\delta_{x_0} + T_m\!\left(\frac{-1}{2\sqrt{n-2}} \right)P_1\delta_{x_0}
\right](x).
\end{align*}
Equation \eqref{eq. bm complete graph} follows by observing that
\[
P_0\delta_{x_0}(x_\ell)=\frac1n,
\qquad
P_1\delta_{x_0}(x_\ell)=\delta_{\ell,0}-\frac1n.
\]
Now we deduce \eqref{eq. n0 complete graph}. Taking $\ell=0$ in \eqref{eq. bm complete graph} gives
\[
b_m(x_0)
=
\frac{2(n-2)^{m/2}}{n}
\left[
T_m\!\left(\frac{n-1}{2\sqrt{n-2}}\right)
+
(n-1)T_m\!\left(-\frac{1}{2\sqrt{n-2}}\right)
\right].
\]
Combining this with Corollary~\ref{cor:Nm_from_bm} with $q=n-2$ yields \eqref{eq. n0 complete graph} for $m\ge2$. For $m=1$, the right-hand side of \eqref{eq. n0 complete graph} becomes
\[
\begin{split}
\frac{2\sqrt{n-2}}{n}
\left[
\frac{n-1}{2\sqrt{n-2}}
+
(n-1)\left(-\frac{1}{2\sqrt{n-2}}\right)
\right]
=0.
\end{split}
\]
Given that there are no closed non-backtracking walks of
length $1$ in $K_n$, $N_1(x_0)=0$, hence the identity \eqref{eq. n0 complete graph} holds true for all $m\ge1$.
\end{proof}

Using the eigenfunctions of the combinatorial Laplacian on $K_n$, Lemma~\ref{cor:Kn_bm_Nm} yields the following identities giving closed summation formulas for sums of products of Chebyshev polynomials and discrete $I$-Bessel functions.

\begin{corollary}\label{cor:Kn_c_identities}
Let $n\ge3$. Then, for every $k\in\mathbb N_0$ and every real number $c$, one has
\begin{equation}\label{eq:Kn_first_identity}
\begin{split}
\sum_{m=1}^{k}
(-1)^m
&\left(
T_m\!\left(\frac{n-1}{2\sqrt{n-2}}\right)
-
T_m\!\left(-\frac{1}{2\sqrt{n-2}}\right)
\right)
I_m^{c}(k)\\
&\qquad =
\frac12
\left[
\left(1-\frac{(n-1)c}{2\sqrt{n-2}}\right)^k
-
\left(1+\frac{c}{2\sqrt{n-2}}\right)^k
\right],
\end{split}
\end{equation}
and 
\begin{equation}\label{eq:Kn_second_identity}
\begin{split}
\sum_{m=1}^{k}&
(-1)^m
\left(
T_m\!\left(\frac{n-1}{2\sqrt{n-2}}\right)
+
(n-1)T_m\!\left(-\frac{1}{2\sqrt{n-2}}\right)
\right)
I_m^{c}(k)\\
&\quad =
\frac12
\left[
\left(1-\frac{(n-1)c}{2\sqrt{n-2}}\right)^k
+
(n-1)
\left(1+\frac{c}{2\sqrt{n-2}}\right)^k
\right]
-\frac{n}{2}I_0^c(k).
\end{split}
\end{equation}
\end{corollary}

\begin{proof}
An orthonormal basis of eigenfunctions of the combinatorial Laplacian on
$K_n$ is given by
\[
\psi_j(x_\ell)=\frac1{\sqrt n}e^{2\pi i j\ell/n},
\qquad j,\ell=0,\ldots,n-1.
\]
Here $\psi_0$ corresponds to the eigenvalue $0$, while
$\psi_j$, $j=1,\ldots,n-1$, correspond to the eigenvalue $n$. Hence the
eigenvalues of $\Delta_{K_n}^{a,b}=b\Delta_{K_n}-aI$ are
\[
\mu_0=-a,\qquad
\mu_j=bn-a,\quad j=1,\ldots,n-1.
\]
Moreover,
\[
\psi_j(x_0)\overline{\psi_j(x_\ell)}
=
\frac1n e^{-2\pi i j\ell/n}.
\]
Substituting this into \eqref{eq. trace fla} with $x=x_\ell$ and $q=n-2$ yields that
for every $k\in\mathbb N_0$, $a\in\mathbb R$, $b\in\mathbb R\setminus \{0\}$ satisfying \eqref{eq:edge} and $\ell=0,\ldots,n-1$,
\[
\sum_{j=0}^{n-1}
e^{-2\pi i j\ell/n}\mu_j^k=n(b(n-1)-a)^k
\sum_{m=0}^{k}
(-1)^m
b_m(x_\ell)(n-2)^{-m/2}
I_m^{c_{a,b}(n)}(k).
\]
Substituting the coefficients $b_m(x_\ell)$ from \eqref{eq. bm complete graph}
into this identity and simplifying gives that
\begin{multline}\label{eq. ident for kn}
\sum_{j=0}^{n-1}
e^{-2\pi i j\ell/n}
\mu_j^k
=
n(b(n-1)-a)^k
\Bigg[
\delta_{\ell,0} I_0^{c_{a,b}(n)}(k) \\
\quad+
\frac{2}{n}
\sum_{m=1}^{k}
(-1)^m
\left(
T_m\!\left(\frac{n-1}{2\sqrt{n-2}}\right)
+
\eta_\ell
T_m\!\left(-\frac{1}{2\sqrt{n-2}}\right)
\right)
I_m^{c_{a,b}(n)}(k)
\Bigg].
\end{multline}
For $\ell\in\{1,\ldots,n-1\}$, we have $\eta_\ell=-1$. Moreover,
\[
\sum_{j=1}^{n-1}e^{-2\pi ij\ell/n}
=-1.
\]
Since $\mu_0=-a$, $\mu_j=bn-a$ for $j=1,\ldots,n-1$, it follows that
\[
\sum_{j=0}^{n-1}
e^{-2\pi ij\ell/n}\mu_j^k
=
(-a)^k-(bn-a)^k.
\]
Taking $\ell\in\{1,\ldots,n-1\}$ in \eqref{eq. ident for kn}, we obtain
\[
\sum_{m=1}^{k}
(-1)^m
\left(
T_m\!\left(\frac{n-1}{2\sqrt{n-2}}\right)
-
T_m\!\left(-\frac{1}{2\sqrt{n-2}}\right)
\right)
I_m^{c_{a,b}(n)}(k)
=
\frac{(-a)^k-(bn-a)^k}{2(b(n-1)-a)^k}.
\]
Taking $\ell=0$ in \eqref{eq. ident for kn} and observing that $\eta_0=n-1$, and
\[
\sum_{j=0}^{n-1}\mu_j^k
=
(-a)^k+(n-1)(bn-a)^k,
\]
we obtain
\begin{multline*}
\sum_{m=1}^{k}
(-1)^m
\left(
T_m\!\left(\frac{n-1}{2\sqrt{n-2}}\right)
+
(n-1)T_m\!\left(-\frac{1}{2\sqrt{n-2}}\right)
\right)
I_m^{c_{a,b}(n)}(k)\\
=
\frac{(-a)^k+(n-1)(bn-a)^k}{
2(b(n-1)-a)^k}
-\frac{n}{2}I_0^{c_{a,b}(n)}(k).
\end{multline*}
Now put $D=b(n-1)-a$ and $c=\frac{2b\sqrt{n-2}}{D}$. Then
\[
-a
=
D\left(1-\frac{(n-1)c}{2\sqrt{n-2}}\right),
\qquad
bn-a
=
D\left(1+\frac{c}{2\sqrt{n-2}}\right).
\]
Substituting these relations into the last two displayed identities gives the stated formulas valid in the range of $c$ for which $a,\, b$ satisfy \eqref{eq:edge}. This range contains infinitely many $c$, and hence extension to all real $c$ follows from the fact that both sides of equations \eqref{eq:Kn_first_identity} and \eqref{eq:Kn_second_identity}  are polynomials in $c$.
\end{proof}
\noindent
\textbf{Proof of \eqref{eq. kn illustration}}.
The identity  follows by subtracting
\eqref{eq:Kn_first_identity} from \eqref{eq:Kn_second_identity}, dividing by
$n$, and using that $T_m(-x)=(-1)^mT_m(x)$.

\subsection{$d$-dimensional hypercube $Q_d$}

Let $d\ge 2$. The $d$-dimensional hypercube $Q_d$ is a $d$-regular graph with $|V(Q_d)|=2^d$ vertices. We identify its vertex set with $\{0,1\}^d$. The eigenvalues of the combinatorial Laplacian on $Q_d$ are $\lambda_j=2j$, $ j=0,1,\ldots,d$, with multiplicity $\binom{d}{j}$. Hence the eigenvalues of $\Delta_{Q_d}^{a,b}=b\Delta_{Q_d}-aI$ are
\[
2bj-a,
\qquad j=0,1,\ldots,d,
\]
with multiplicity $\binom{d}{j}$.

We fix $x_0=(0,\ldots,0)$. In the following lemma, for all $x\in V(Q_d)$ we explicitly compute the combinatorial quantities $b_m(x)$ and $N_m(x_0)$.

\begin{lemma}\label{cor:Qd_bm_Nm}
Let $Q_d$ be the $d$-dimensional hypercube. For $x\in V(Q_d)$, let $r=d(x_0,x)=|x|$.
For every $m\ge1$, 
\begin{equation}\label{bm qd}
b_m(x)
=
\frac{(d-1)^{m/2}}{2^{d-1}}
\sum_{j=0}^{d}
K_j(r;d)
T_m\!\left(
\frac{d-2j}{2\sqrt{d-1}}
\right),
\end{equation}
where $K_j(r;d)$ is the $j$-th Krawtchouk polynomial.

In particular, for every $m\ge 1$, the number $N_m(x_0)$ of closed
non-backtracking walks of length $m$ without tails starting at $x_0$ is
\begin{equation}\label{Q_d_formula}
N_m(x_0) =
\frac{(d-1)^{m/2}}{2^{d-1}}
\sum_{j=0}^{d}
\binom{d}{j}
T_m\!\left(
\frac{d-2j}{2\sqrt{d-1}}
\right) 
+
\frac{1+(-1)^m}{2}(d-2).
\end{equation}
\end{lemma}

\begin{proof}
For
$\alpha=(\alpha_1,\ldots,\alpha_d)\in\{0,1\}^d$, define
\[
\psi_\alpha(x)=(-1)^{\alpha\cdot x},
\qquad
\alpha\cdot x=\alpha_1x_1+\cdots+\alpha_dx_d .
\]
Then the family $\{2^{-d/2}\psi_\alpha\}_{\alpha\in\{0,1\}^d}$ is an orthonormal basis of eigenfunctions of the adjacency operator on
$Q_d$. A direct computation shows that
\[
A\psi_\alpha=(d-2|\alpha|)\psi_\alpha,
\]
where $|\alpha|=\alpha_1+\cdots+\alpha_d$.
For $m\ge1$, by Lemma~\ref{lemma:bm_from_cm}
\[
b_m(x)
=
\left[
2(d-1)^{m/2}
T_m\!\left(\frac{A}{2\sqrt{d-1}}\right)\delta_{x_0}
\right](x).
\]
Using the spectral decomposition of $A$ with respect to the basis
$\{2^{-d/2}\psi_\alpha\}_{\alpha\in\{0,1\}^d}$, we obtain
\[
b_m(x)
=
\frac{(d-1)^{m/2}}{2^{d-1}}
\sum_{\alpha\in\{0,1\}^d}
T_m\!\left(
\frac{d-2|\alpha|}{2\sqrt{d-1}}
\right)
(-1)^{\alpha\cdot x}.
\]
Grouping the terms according to $j=|\alpha|$, we obtain
\[
b_m(x)
=
\frac{(d-1)^{m/2}}{2^{d-1}}
\sum_{j=0}^{d}
T_m\!\left(
\frac{d-2j}{2\sqrt{d-1}}
\right)
\sum_{|\alpha|=j}
(-1)^{\alpha\cdot x}.
\]
If $|x|=r$, then $\alpha\cdot x$ equals the number of
coordinates at which both $\alpha$ and $x$ are equal to $1$.
If this number is $\ell$, then there are
$\binom r\ell \binom{d-r}{j-\ell}$ such vectors $\alpha$, and in this case $(-1)^{\alpha\cdot x}=(-1)^\ell$. Therefore
\[
\sum_{|\alpha|=j}
(-1)^{\alpha\cdot x}
=
\sum_{\ell=0}^{j}
(-1)^\ell
\binom{r}{\ell}
\binom{d-r}{j-\ell}
=
K_j(r;d),
\] 
where $K_j(r;d)$ is the $j$-th Krawtchouk polynomial of the binary Hamming scheme (see, e.g., \cite[Section 3.2]{BI84}). This proves the formula for $b_m(x)$.

Taking $x=x_0$, we have $r=0$, and hence $K_j(0;d)=\binom{d}{j}$. 
Therefore
\[
b_m(x_0)
=
\frac{(d-1)^{m/2}}{2^{d-1}}
\sum_{j=0}^{d}
\binom{d}{j}
T_m\!\left(
\frac{d-2j}{2\sqrt{d-1}}
\right).
\]
Since $Q_d$ is vertex-transitive and $q=d-1$, Corollary~\ref{cor:Nm_from_bm}
gives the formula for $N_m(x_0)$ for $m\ge2$. For $m=1$, the right-hand
side of \eqref{Q_d_formula} is zero, which agrees with the fact that $Q_d$
has no loops. Hence \eqref{Q_d_formula} holds for all $m\ge1$.
\end{proof}

As an immediate consequence of the bipartite structure of the hypercube,
together with \eqref{Q_d_formula} and Corollary~\ref{cor:Nm_from_bm}, one obtains the following identity for Chebyshev polynomials.

\begin{corollary}
For every $d\ge2$ and every $m\ge0$, one has
\[
\sum_{j=0}^{d}
\binom{d}{j}
T_{2m+1}\!\left(
\frac{d-2j}{2\sqrt{d-1}}
\right)
=0.
\]
\end{corollary}


\begin{remark}
The identity of the previous corollary can also be derived directly from
the spectrum of the hypercube. Indeed, the adjacency eigenvalues of $Q_d$
are $\theta_j=d-2j$, $j=0,1,\ldots,d$, with multiplicity $\binom{d}{j}$. Since
\[
\theta_{d-j}=-(d-2j)=-\theta_j \quad \text{and} \quad \binom{d}{d-j}=\binom{d}{j},
\]
the spectrum is symmetric with respect to the origin. Moreover,
$T_{2m+1}$ is an odd polynomial. Consequently, the terms corresponding to $j$ and $d-j$ cancel pairwise, which yields the identity.
\end{remark}

Using the formula for $b_m(x)$ in \eqref{eq. trace fla} and
grouping the spectral side by Hamming weights, we obtain the following
twisted Krawtchouk--Chebyshev identity.

\begin{corollary}\label{cor:Qd_twisted}
Let $d\ge2$ and let $1\le r\le d$. Then, for every $k\in\mathbb N$ and every real number $c$, one has
\begin{equation} \label{Qd_twisted}
    \sum_{j=0}^{d}
K_j(r;d)
\left(
1+\frac{2j-d}{2\sqrt{d-1}}\,c
\right)^k
=
2
\sum_{m=1}^{k}
(-1)^m
\sum_{s=0}^{d}
K_s(r;d)
T_m\!\left(
\frac{d-2s}{2\sqrt{d-1}}
\right)
I_m^{c}(k).
\end{equation}
\end{corollary}

\begin{proof}
We want to apply \eqref{eq. trace fla} to the hypercube $Q_d$. Choose
$x\in V(Q_d)$ such that $|x|=r$.  For any
$\alpha\in\{0,1\}^d$, the corresponding eigenvalue of
$\Delta_{Q_d}^{a,b}$ is $\mu_\alpha=2b|\alpha|-a$. Therefore, after multiplying
by $2^d$, the spectral side of \eqref{eq. trace fla} becomes
\[
\sum_{\alpha\in\{0,1\}^d}
(2b|\alpha|-a)^k(-1)^{\alpha\cdot x}
=
\sum_{j=0}^{d}
(2bj-a)^kK_j(r;d),
\qquad k\ge0.
\]
Hence \eqref{eq. trace fla} yields 
\[
\sum_{j=0}^{d}
K_j(r;d)(2bj-a)^k
=
2^d(bd-a)^k
\sum_{m=0}^{k}
(-1)^m
b_m(x)
(d-1)^{-m/2}
I_m^{c_{a,b}(d)}(k),
\]
for all $a,\, b$ satisfying \eqref{eq:edge}, where $c_{a,b}(d)=\frac{2b\sqrt{d-1}}{bd-a}$. Since $r>0$, we have
$b_0(x)=0$. Substituting \eqref{bm qd} into the above identity and simplifying, we obtain
\[
\sum_{j=0}^{d}
K_j(r;d)(2bj-a)^k
=
2(bd-a)^k
\sum_{m=1}^{k}
(-1)^m
\sum_{s=0}^{d}
K_s(r;d)
T_m\!\left(
\frac{d-2s}{2\sqrt{d-1}}
\right)
I_m^{c_{a,b}(d)}(k).
\]
Dividing by $(bd-a)^k$ and introducing $c=\frac{2b\sqrt{d-1}}{bd-a}$, we obtain
\[
2bj-a
=
(bd-a)
\left(
1+\frac{2j-d}{2\sqrt{d-1}}\,c
\right),
\]
and therefore \eqref{Qd_twisted} holds true for all $c=\frac{2b\sqrt{d-1}}{bd-a}$ for which $a,\, b$ satisfy \eqref{eq:edge}. This range contains infinitely many $c$, and hence extension to all real $c$ follows from the fact that both sides of the identity  \eqref{Qd_twisted} are polynomials in $c$.
\end{proof}

\end{document}